\documentclass{article}

\usepackage{hyperref}
\usepackage{mystyle}
\usepackage[accepted]{icml2012}

\usepackage{enumitem}
\icmltitlerunning{Degrees of Freedom of the Group Lasso}

\begin{document} 

\twocolumn[
\icmltitle{Degrees of Freedom of the Group Lasso}
\icmlauthor{Samuel~Vaiter, Charles~Deledalle, Gabriel~Peyr\'{e}}{vaiter@ceremade.dauphine.fr}
\icmladdress{CEREMADE, CNRS, Universit\'{e} Paris-Dauphine, France}
\icmlauthor{Jalal~Fadili}{}
\icmladdress{GREYC, CNRS-ENSICAEN-Universit\'{e} de Caen, France}
\icmlauthor{Charles~Dossal}{}
\icmladdress{IMB, Universit\'{e} Bordeaux 1, France}

\icmlkeywords{sparsity, group lasso, degrees of freedom, sensitivity analysis, ICML}

\vskip 0.3in
]

\begin{abstract}
This paper studies the sensitivity to the observations of the block/group Lasso solution to an overdetermined linear regression model.
Such a regularization is known to promote sparsity patterns structured as nonoverlapping groups of coefficients.
Our main contribution provides a local parameterization of the solution with respect to the observations.
As a byproduct, we give an unbiased estimate of the degrees of freedom of the group Lasso.
Among other applications of such results, one can choose in a principled and objective way the regularization parameter of the Lasso through model selection criteria.
\end{abstract}

\section{Introduction} % (fold)
\label{sec:introduction}

This paper deals with the overdetermined linear regression model of the form $y = \XX \beta_0 + \epsilon$ where $y \in \RR^Q$ is the observation/response vector, $\beta_0 \in \RR^N$ the regression vector, $\XX$ is the design matrix whose columns are linearly independent, and $\epsilon$ is an additive noise.
Note that $Q > N$ and $\XXX$ is an invertible matrix.

\subsection{Group Lasso} % (fold)
\label{sub:group}

A block segmentation $\Bb$ corresponds to a disjoint union of the set of indices i.e.~$\bigcup_{b \in \Bb} = \ens{1,\ldots,N}$ and for each $b, b' \in \Bb, b \cap b' = \emptyset$.
For $\beta \in \RR^N$, for each $b \in \Bb$, $x_b=(\beta_i)_{i \in b}$ is a subvector of $\beta$ whose entries are indexed by the block $b$, where $|b|$ is the cardinality of $b$. 

We consider the Group Lasso or Block Sparse regularization introduced by \cite{bakin1999adaptive,yuan2006model} which reads
\begin{equation}\label{eq-group-lasso}\tag{$\lasso$}
  \umin{ \beta \in \RR^N }  
  \frac{1}{2}\norm{y - \XX \beta}^2 
  + \lambda \sum_{b \in \Bb} \norm{\beta_b} ,
\end{equation}
where $\lambda > 0$ is the so-called regularization parameter and $\norm{\cdot}$ is the $\ldeux$-norm.
Note that if each block $b$ is of size 1, we recover the standard Lasso \cite{tibshirani1996regre}.

% subsection group (end)

\subsection{Degrees of Freedom} % (fold)
\label{sub:dof}

We focus in this paper on the variations of the solution $\xsoly(y)$ of $\lasso$ with respect to the observations $y$.
This turns out to be a pivotal ingredient to compute the effective degrees of freedom (DOF) usually used to quantify the complexity of a statistical modeling procedure.

Let $\hat{\mu}(y)=\XX \xsoly(y)$ be the response or the prediction associated to the estimator $\xsoly(y)$ of $\beta_0$, and let $\mu_0 = \XX \beta_0$.
It is worth noting that $\hat{\mu}(y)$ is always uniquely defined, although when $\xsoly(y)$ is not as is the case of rank-deficient or underdetermined design matrix $X$.
Note that any estimator $\hat{\mu}$ of $\mu_0$ might be considered.
We also make the assumption that $\epsilon$ is an additive white Gaussian noise term $\epsilon \sim \Nn(0,\sigma^2 I_Q)$, hence $y$ follows the law $\Nn(\mu_0,\sigma^2 I_Q)$ and according to \cite{efron1986biased}, the DOF is given by
\begin{equation*}
\DOF = 
\sum_{i=1}^Q \frac{\mathrm{cov}(y_i, [\hat{\mu}(y)]_i)}{\sigma^2} ~.
\end{equation*}
The well-known Stein's lemma asserts that if $\hat{\mu}$ is weakly differentiable then its divergence is an unbiased estimator of its DOF, i.e.
\begin{equation*}
  \hat{\DOF} = \tr(\partial \hat{\mu}(y)) \qandq \EE_\epsilon(\hat{\DOF}) = \DOF.
\end{equation*}
An unbiased estimator of the DOF provides an unbiased estimate for the prediction risk of $\hat{\mu}(y)$ through e.g. the Mallow's $C_p$ \cite{mallows1973some}, the AIC \cite{akaike1973information}, the SURE \cite{stein1981estimation} or the GCV \cite{golub1979generalized}.
These quantities can serve as model selection criteria to assess the accuracy of a candidate model.

% subsection dof (end)

\subsection{Previous Works} % (fold)
\label{sub:pw}

In the special case of standard Lasso with a linearly independent design, \cite{zou2007degrees} show that the number of nonzero coefficients is an unbiased estimate for the degrees of freedom.
This work is extended in \cite{DOSSAL-2011-682903} to an arbitrary design matrix.
The DOF of the analysis sparse regularization (a.k.a.~generalized Lasso in statistics) is studied in \cite{tibshirani2012dof,vaiter-local-behavior}.
A formula of an estimate of the DOF for the group Lasso when the design is orthogonal within each group is conjectured in \cite{yuan2006model}.
\cite{kato2009degrees} studies the DOF of a generalization of the Lasso where now the regression coefficients are constrained to a closed convex set.
He provides an unbiased estimate of the DOF for the constrained version of the group Lasso under the same orthogonality assumption on $\XX$ as \cite{yuan2006model}.
An estimate of the DOF for the group Lasso is also given in \cite{solo2010threshold} by an heuristic proof in the full column rank case, but its unbiasedness is not proved.

% subsection pw (end)

\subsection{Contributions} % (fold)
\label{sub:contributions}

This paper proves a general result (Theorem \ref{thm-local}) on the variations of the solutions to \lasso with respect to the observation/response vector $y$.
With such a result on hand, Theorem \ref{thm-dof} provides a provably unbiased estimate of the DOF.
These contributions are detailed in Sections \ref{sub:contrib-local} and \ref{sub:contrib-dof} below. The proofs are deferred to Section \ref{sec:proofs} awaiting inspection by the interested reader.

\subsection{Notations} % (fold)
\label{sub:notations}

% subsection notations (end)
We start by some notations used in the sequel.
We extend the notion of support, commonly used in sparsity by defining the \bsup $\bs(\beta)$ of $\beta \in \RR^N$ as
\begin{equation*}
  \bs(\beta) = \enscond{b \in \Bb}{\norm{\beta_b} \neq 0}.
\end{equation*}
The size of $\bs(\beta)$ is defined as $\abs{\bs(\beta)} = \sum_{b \in \Bb} \abs{b}$.
We denote by $\XX_{I}$, where $I$ is a \bsup, the matrix formed by the columns $\XX_i$ where $i$ is an element of $b \in I$.
We introduce the following block-diagonal operator
\begin{equation*}
  \delta_\beta : v \in \RR^{|I|} \mapsto ( v_b / \norm{\beta_b} )_{b \in \Bb} \in \RR^{|I|} .
\end{equation*}
and
\begin{equation*}
  P_\beta : v \in \RR^{|I|} \mapsto ( P_{\beta_b^\bot}( v_b ) )_{ b \in \Bb } \in \RR^{|I|}
\end{equation*}
where $P_{\beta_b^\bot}$ is the projector orthogonal to $x_b$.
For any operator $A$, we denote $A^{\ins{T}}$ its adjoint.

% subsection notations (end)

% section introduction (end)

\section{Main results} % (fold)
\label{sec:main_results}

Note that as the $\XX$ is assumed full column rank, $\lasso$ has exactly one global minimizer $\xsoly(y)$.
Hence, we define the single-valued mapping $y \mapsto \xsoly(y)$.

\subsection{Local Parameterization} % (fold)
\label{sub:contrib-local}

Let $I$ be the \bsup of some vector $\beta$.
For any block $b \not\in I$, we define
\begin{equation*}
  \Hh_{I,b} \!=\! 
  \Big\{ 
     y \!\in\! \RR^Q \backslash 
     \exists \beta  \!:\!
     \forall c \!\in\! I,
     (\norm{ \XX_b^{\ins{T}} r},
      \XX_c^{\ins{T}} r)
      = 
      (\lambda,\lambda \frac{\beta_c}{\norm{\beta_c}})
  \Big\} .
\end{equation*}
where $r = y - \XX_I \beta$.
  
\begin{defn}\label{defn:h}
  The \emph{transition space} $\Hh$ is defined as
  \begin{equation*}
    \Hh = \bigcup_{I \subset \Ii} \bigcup_{b \not\in I} \Hh_{I,b},
  \end{equation*}
  where $\Ii$ is the set of sub-sets of $\{ 0,\ldots,N-1 \}$ obtained as unions of blocks.
\end{defn}

We prove the following sensitivity theorem
\begin{thm}\label{thm-local}
  Let $y \not\in \Hh$, and $I = \bs(\xsoly(y))$ the \bsup of $\xsoly(y)$.
  There exists a neighborhood $\Oo$ of $y$ such that
  \begin{enumerate}
    \item the \bsup of $\xsoly(y)$ is constant on $\Oo$, i.e.~
    \begin{equation*}
      \forall y \in \Oo, \quad \bs(\xsoly(\bar y)) = I,
    \end{equation*}
    \item the mapping $\xsoly$ is $\Cc^1$ on $\Oo$ and its differential is such that
      \begin{equation}\label{eq-differential}
        [\partial \xsoly(\bar y)]_{I^c} = 0 \qandq [\partial \xsoly(\bar y)]_I = d(y) ,
      \end{equation}
    where
    \begin{equation*}
      d(y) = \bpa{ \XXX + \lambda \delta_{\xsoly(y)}  \circ P_{\xsoly(y)}  }^{-1} \XX_I^{\ins{T}}
    \end{equation*}
    and
    \begin{equation*}
      I^c = \enscond{b \in \Bb}{b \notin I}.
    \end{equation*}
  \end{enumerate}
\end{thm}

% subsection local (end)

\subsection{Degrees of Freedom} % (fold)
\label{sub:contrib-dof}

We consider the estimator $\hat{\mu}(y) = \XX \xsoly(y)$.
\begin{thm}\label{thm-dof}
  Let $\lambda > 0$.
  The mapping $y \mapsto \hat{\mu}(y)$ is of class $\Cc^1$ on $\RR^Q \setminus \Hh$ and,
  \begin{equation}\label{eq:diverg}
    \diverg(\hat{\mu}(y)) = \tr (\XX_I d(y)) ,
  \end{equation}
  where $\xsoly(y)$ is the solution of $\lasso$ and $I = \bs(\xsoly(y))$.
  Moreover, the set $\Hh$ has zero Lebesgue measure, thus if $\epsilon \sim \Nn(0,\sigma^2 I_Q)$, equation \eqref{eq:diverg} is an unbiased estimate of the DOF of the group Lasso.
\end{thm}
%This theorem shows that $\hat{\DOF} = \XX_I d(y)$ is an unbiased estimate of the degrees of freedom.
We specify this result for the Block Soft Thresholding
\begin{cor}\label{cor-particular}
  If $\XX = \Id$, one has
  \begin{equation*}
    \hat{\DOF} = \abs{J} - \lambda \sum_{b \subset J} \dfrac{\abs{b} - 1}{\norm{y_b}}
  \end{equation*}
  where $J = \bigcup \enscond{b \in \Bb}{\norm{y_b} > \lambda}$.
\end{cor}
% subsection dof (end)

% section main_results (end)

\section{Proofs} % (fold)
\label{sec:proofs}

This section details the proofs of our results.
We introduce the following normalization operator
\begin{equation*}
  \no(\beta_I) = v \qwhereq \forall b \in I, v_b = \frac{\beta_b}{\norm{\beta_b}} .
\end{equation*} 
We use the following lemma in our proofs which is a straightforward consequence of the first order necessary and sufficient condition of a minimizer of the group Lasso problem $\lasso$.
\begin{lem}\label{lem:first-order}
  A vector $\beta \in \RR^N$ is the solution of \lasso if, and only if, these two conditions holds
  \begin{enumerate}
    \item On the \bsup $I = \bs(\beta)$,
    \begin{equation*}
      \XX_I^{\ins{T}}(y - \XX_I \beta_I) = \lambda \no(\beta_I) .
    \end{equation*}
    \item For all $b \in \Bb$ such that $b \not\in I$, one has
    \begin{equation*}
      \norm{\XX_b^{\ins{T}}(y - \XX_I \beta_I)} \leq \lambda .
    \end{equation*}
  \end{enumerate}
\end{lem}
A proof of this lemma can be found in \cite{bach2008consistency}.

\subsection{Proof of Theorem \ref{thm-local}} % (fold)
\label{sub:local}

We will use this following lemma.
\begin{lem}\label{lem-invertible}
  Let $\beta \in \RR^N$ and $\lambda > 0$.
  Then $\XXX+ \lambda \delta_\beta \circ P_\beta$ is invertible.
\end{lem}
\begin{proof}
We prove that $\XXX + \lambda \delta_\beta \circ P_\beta$ is symmetric positive definite.
Remark that $\XXX$ is a positive definite matrix.
Moreover, $\delta_\beta \circ P_\beta$ is symmetric positive semi-definite since both $\delta_\beta$ and $P_\beta$ are SDP and commute.
We conclude using the fact that the sum of a symmetric positive definite matrix and a symmetric positive semi-definite matrix is symmetric positive definite.
\end{proof}

Let $y \not\in \Hh$.
We define $I = \bs(\xsoly(y))$ the \bsup of the solution $\xsoly(y)$ of $\lasso$.
We define the following mapping 
\begin{equation*}
  \Gamma(\alpha_I,y) = \XX_I^{\ins{T}}(\XX_I \alpha_I - y) + \lambda \no(\alpha_I) .
\end{equation*}
Observe item 1 of Lemma \ref{lem:first-order} is equivalent to $\Gamma([\xsoly(y)]_I,y) = 0$.

Our proof is done in three steps.
We first (\textbf{1.}) prove there exists a mapping $\bar y \mapsto \solSS$ such that for every element of a neighborhood of $y$ one has $\Gamma([\solSS]_I, \bar y) = 0$ and $[\solSS]_{I^c} = 0$.
Then, we prove (\textbf{2.}) that $\solSS = \xsoly(\bar y)$ is the solution of $\lassoP{\bar y}{\lambda}$ in a neighborhood of $y$.
Finally, we obtain (\textbf{3.}) equation \eqref{eq-differential} from the implicit function theorem.

\textbf{1.} The derivative of $\Gamma$ with respect to the first variable reads on $(\RR^{\abs{I}} \setminus U) \times \RR^Q$
\begin{equation*}
    \partial_1 \Gamma(\beta_I,y) = \XXX + \lambda \delta_{\beta_I} \circ P_{\beta_I} .
\end{equation*}
where $U = \enscond{\alpha \in \RR^{\abs{I}}}{\exists b \in I: \alpha_b = 0}$.
The mapping $\partial_1 \Gamma$ is invertible according to Lemma \ref{lem-invertible}.
Hence, using the implicit function theorem, there exists a neighborhood $\tilde \Oo$ of $y$ such that we can define a mapping $\beta_I : \Oo \to \RR^{\abs{I}} $ of class $\Cc^1$ over $\tilde \Oo$ that satisfies for $\bar y \in \tilde \Oo$
\begin{equation*}
    \Gamma(\beta_I(\bar y),\bar y)  =  0
    \qandq
    \beta_I(y) = [\xsoly(y)]_I.
\end{equation*}
We then extend $\beta_I$ on $I^c$ as $\beta_{I^c}(\bar y)=0$, which defines a mapping $\solSS : \tilde \Oo \to \RR^N$.

\textbf{2.} Writting the first-order conditions on $\xsoly(y)$ on the blocks not included in the \bsup, one has
\begin{equation*}
  \forall b \notin I, \quad \norm{\XX_b^{\ins{T}}(y - \XX_I [\xsoly(y)]_I)} \leq \lambda .
\end{equation*}
Suppose there exists $b \notin I$ such that $\norm{ \XX_b^{\ins{T}}( \XX_I [\xsoly(y)]_I - y ) } = \lambda$.
Then $y \in \Hh_{I,b}$ since
\begin{equation*}
  \norm{\XX_b^{\ins{T}} r} = \lambda \qandq \XX_c^{\ins{T}} r = \lambda \dfrac{[\xsoly(y)]_c}{\norm{[\xsoly(y)]_c}} ,
\end{equation*}
for $r = y - \XX_I [\xsoly(y)]_I$, which is a contradiction with $y \notin \Hh$.
Hence,
\begin{equation*}
    \foralls b \notin I, \quad
    \norm{ \XX_b^{\ins{T}}( \XX_I [\xsoly(y)]_I - y ) } < \lambda .   
\end{equation*}
By continuity of $\bar y \mapsto \beta_I(\bar y)$ and since $\beta_I(y) = [\xsoly(y)]_I$, we can find a neighborhood $\Oo$ included in $\tilde \Oo$ such that for every $\bar y \in \Oo$, one has
\begin{equation*}
    \foralls b \notin I, \quad
    \norm{ \XX_b^{\ins{T}}( \XX_I \beta_I(\bar y) - \bar y ) } \leq \lambda .
\end{equation*}
Moreover, by definition of the mapping $\beta_I$, one has
\begin{equation*}
  \XX_I^{\ins{T}}(y - \XX_I \beta_I(\bar y)) = \lambda \no(\beta_I(\bar y)) \text{ and } \bs(\beta_I(\bar y)) = I .
\end{equation*}
According to Lemma \ref{lem:first-order}, the vector $\beta(\bar y)$ is solution of $\lassoP{\bar y}{\lambda}$.
Since $\lassoP{\bar y}{\lambda}$ admits a unique solution, $\xsoly(\bar y) = \beta(\bar y)$ for every $\bar y \in \Oo$.

\textbf{3.} Using the implicit function theorem, one obtains the derivative of $[\solSp]_I$ as
\begin{equation*}
    [\partial \xsoly(\bar y)]_I  = -
    \big( \partial_1 \Gamma([\xsoly(y)]_I,y) \big)^{-1} 
    \circ 
    \big( \partial_2 \Gamma([\xsoly(y)]_I,y) \big)        
\end{equation*}
where $\partial_2 \Gamma([\xsoly(y)]_I,y)= - \XX_I^{\ins{T}}$, which leads us to \eqref{eq-differential}.

% subsection local (end)
\subsection{Proof of Theorem \ref{thm-dof}} % (fold)
\label{sub:dof}

We define for each $b \in \Bb$
\begin{align*}
  \Hh_{I,b}^\uparrow = \Big \{ &(r,\beta)  \in \RR^Q \times \RR^{|I|}  \, \, \backslash \\ 
    & \norm{ \XX_b^{\ins{T}}r } =  1 \qandq
    \forall g \in I, \quad \XX_g^{\ins{T}} r = \dfrac{\beta_g}{\norm{\beta_g}} \Big \}
    .
\end{align*}
We prove (\textbf{1.}) that for each $b \in I$, $\Hh_{I,b}^\uparrow$ is a manifold of dimension $Q - 1$.
Then (\textbf{2.}) we prove that $\Hh$ if of measure zero with respect to the Lebesgue measure on $\RR^Q$.
Finally (\textbf{3.}), we prove that $\hat{\DOF}$ is an unbiased estimate of the DOF.

\textbf{1.} Note that $\Hh_{I,b}^\uparrow = \psi^{-1}(\ens{0})$ where
\begin{equation*}
  \psi(r,\beta) \!=\! \left(\!\norm{\XX_b^{\ins{T}}r}^2 \!-\! 1,
  \XX_I^{\ins{T}} r - \no(\beta) \right) .
\end{equation*}
Remark that the adjoint of the differential of $\psi$
\begin{equation*}
   (\partial \psi)^{\ins{T}}(r,\beta) = 
   \left(
   \begin{array}{c|c}
    2 \XX_b^{}  \XX_b^{\ins{T}} r & \XX_I \\ \hline
     0 & \lambda \delta_\beta \circ P_\beta
   \end{array}
   \right)
\end{equation*}
has full rank.
Indeed, consider the matrix $A = (2 \XX_b^{}  \XX_b^{\ins{T}} r | \XX_I)$.
Let $\alpha  = (s,u)^{\ins{T}} \in \RR \times \RR^{\abs{I}}$ such that $A \alpha = 0$.
Then $(2s \XX_b r, u)^{\ins{T}} \in \Ker (\XX_b | \XX_I)$.
Since $(\XX_b | \XX_I)$ has full rank, we conclude that $\alpha = 0$.
As a consequence, $\partial \psi(r,\beta)$ is non-degenerated.
Finally, $\Hh_{I,b}^\uparrow$ is a manifold of dimension $Q - 1$.

\textbf{2.} We prove that $\Hh_{I,b}$ is of Hausdorff dimension less or equal to $Q - 1$.
Consider the following mapping
\begin{equation*}
  \phi : 
  \left\{
  \begin{array}{ccc}
    \RR^Q \times \RR^{\abs{I}} & \to & \RR^Q \times \RR^{\abs{I}} \\
    (r,\beta) & \mapsto & (r + \XX_I \beta, \XX_I^{\ins{T}} r)
  \end{array}.
  \right.
\end{equation*}
The mapping $\phi$ is a $\Cc^1$-diffeomorphism between $\RR^Q \times \RR^{\abs{I}}$ and itself.
Thus, $A = \phi(\Hh_{I,b}^\uparrow)$ is a manifold of dimension $Q-1$.
We now introduce the projection
\begin{equation*}
  \pi : 
  \left\{
  \begin{array}{ccc}
    A & \to & \RR^Q \\
    (y,\alpha) & \mapsto & y
  \end{array} .
  \right.
\end{equation*}
Observe that $\Hh_{I,b} = \pi(A)$.
According to Hausdorff measure properties \cite{rogers1998hausdorff}, since $\pi$ is 1-Lipschitz, the Hausdorff dimension of $\pi(A)$ is less or equal to the Hausdorff dimension of $A$ which is the dimension of $A$ as a manifold, namely $Q-1$.
Hence, the measure of $\Hh_{I,b}$ w.r.t the Lebesgue measure of $\RR^Q$ is zero.

\textbf{3.} According to Theorem \ref{thm-local}, $y \mapsto \xsoly(y)$ is $\Cc^1$ on $\RR^Q \setminus \Hh$.
Composing by $\XX$ gives that $\hat{\mu}$ is differentiable almost everywhere, hence weakly differentiable.
Moreover, taking the divergence of the differential \eqref{eq-differential}, one obtains \eqref{eq:diverg}.
This formula is verified almost everywhere, outside the set $\Hh$.
Stein's Lemma \cite{stein1981estimation} gives the unbiased propertiy of our estimator $\hat \DOF$ of the DOF.

\subsection{Proof of Corollary \ref{cor-particular}} % (fold)
\label{sub:proof-cor}
When $\XX = \Id$,  the solution of $\lasso$ is a block soft thresholding
\begin{equation}\label{eq:block-soft}
  [\xsoly(y)]_b = 
  \begin{cases}
      0 & \text{if } \norm{y_b} \leq \lambda \\
      (1 - \frac{\lambda}{\norm{y_b}}) y_b & \text{otherwise}
    \end{cases} .
\end{equation}
For every $b \in J$, we differentiate equation \eqref{eq:block-soft}
\begin{equation*}
  [\partial \xsoly(y)]_b :  \alpha \in \RR^{\abs{b}} \mapsto \alpha - \frac{\lambda}{\norm{y_b}} P_{y_b^\bot}(\alpha) .
\end{equation*}
Since $P_{y_b^\bot}(\alpha)$ is a projector on space of dimension $\abs{b} - 1$, one has $\tr (P_{y_b^\bot}) = \abs{b} - 1$.

% section proof -cor(end)
\bibliographystyle{icml2012}

\end{document}